\documentclass[12pt]{amsart}
\usepackage{amsmath,amsthm,amsfonts,amssymb,mathrsfs}
\date{\today}

\usepackage{color}
\input xymatrix
\xyoption{all}

\usepackage{hyperref}

\newtheorem{theorem}{Theorem}

\newtheorem{proposition}[theorem]{Proposition}
\newtheorem{corollary}[theorem]{Corollary}
\newtheorem{lemma}[theorem]{Lemma}
\theoremstyle{definition}
\newtheorem{example}[theorem]{Example}
\newtheorem{remark}[theorem]{Remark}

\begin{document}

\title[$H$-closed quasitopological groups]{$H$-closed quasitopological groups}

\author[S.~Bardyla]{Serhii~Bardyla}
\address{Faculty of Mathematics, National University of Lviv,
Universytetska 1, Lviv, 79000, Ukraine}
\email{sbardyla@yahoo.com}

\author[O.~Gutik]{Oleg~Gutik}
\address{Faculty of Mathematics, National University of Lviv,
Universytetska 1, Lviv, 79000, Ukraine}
\email{o\underline{\hskip5pt}\,gutik@franko.lviv.ua,
ovgutik@yahoo.com}

\author[A.~Ravsky]{Alex~Ravsky}
\address{Department of Functional Analysis, Pidstryhach Institute for Applied Problems of Mechanics and Mathematics
National Academy of Sciences of Ukraine, Naukova 2-b, Lviv, 79060, Ukraine}
\email{oravsky@mail.ru}



\begin{abstract}
An $H$-closed quasitopological group is a Hausdorff
quasitopological group which is contained in each
Hausdorff quasitopological group as a closed subspace.
We obtained a sufficient condition for a quasitopological group
to be $H$-closed, which allowed us to
solve a problem by Arhangel'skii and Choban and to show that a topological group $G$ is
$H$-closed in the class of quasitopological groups if and only if $G$ is Ra\v\i kov-complete.
Also we present examples of non-compact quasitopological groups whose topological spaces are $H$-closed.
\end{abstract}

\maketitle

One of the functions of a theory is to establish a correspondence between outer relations of an
object with other objects and inner properties of the object. Now we proceed to concrete objects from
general topology and topological algebra.

Further we shall follow the terminology of \cite{ArhangelskiiTkachenko2008, BerglundJunghennMilnes1989, Engelking1989}.
In this paper term ``space'' means a Hausdorff topological space.

If $Y$ is a subspace of a topological space $X$ and $A\subseteq Y$, then by $\operatorname{cl}_Y(A)$ and $\operatorname{int}_Y(A)$ we denote closure and interior of $A$ in $Y$, respectively. By $\mathbb{R}$, $\mathbb{Q}$ and $\mathbb{N}$ we denote the set of real, rational, and positive integer numbers, respectively.

It is well known that a compact space $X$ is a closed subspace of any Hausdorff space which contains $X$.
So we define a Hausdorff space $X$ to be {\it $H$-closed} provided $X$ is a closed subspace of any Hausdorff  space which
contains $X$. So $H$-closedness is an outer relation of a space. But it turned to be
equivalent to an inner property of a space.

\begin{theorem}[{\cite[Exercise~3.12.5]{Engelking1989}, \cite{AlexandroffUrysohn1929}, (announcement in
\cite{AlexandroffUrysohn1923} and \cite{Aleksandrov1939}})]\label{HClSp}
For a Hausdorff space $X$ the following conditions are equivalent:
\begin{itemize}
  \item[(1)] the space $X$ is $H$-closed;
  \item[(2)] for every family $\left\{V_s\colon s\in S\right\}$ of open subsets of $X$ which has the finite intersection property the intersection $\bigcap\left\{\operatorname{cl}_X(V_s)\colon s\in S\right\}$ is non-empty;
  \item[(3)] every ultrafilter in the family of all open subsets of $X$ converges;
  \item[(4)] every open cover $\left\{U_s\colon s\in S\right\}$ of the space $X$ contains a finite subfamily $\left\{U_{s_1},U_{s_2},\ldots, U_{s_n}\right\}$ such that $\operatorname{cl}_X\left(U_{s_1}\right)\cup \operatorname{cl}_X\left(U_{s_2}\right)\cup\dots\cup\operatorname{cl}_X\left(U_{s_n}\right)=X$.
\end{itemize}
\end{theorem}

A regular space is $H$-closed if and only if it is compact, but there exists a non-regular
$H$-closed space (see \cite{AlexandroffUrysohn1923}).

A \emph{semitopological group} consists of a group $G$ and a topology $\tau$ on the set $G$ such that the group operation $\cdot\colon G\times G\to G$ is separately continuous. A semitopological group with continuous inversion $\operatorname{\textbf{inv}}\colon G\rightarrow G\colon x\mapsto x^{-1}$ is called a \emph{quasitopological group}. Also, a semitopological group (resp., a quasitopological group) with continuous group operation is called a \emph{paratopological group} (resp., a \emph{topological group}).

We shall say that a semitopological group $G$ is \emph{$H$-closed} in a class of Hausdorff semitopological groups $\mathscr{S}$ if $G$ is a closed subgroup of every semitopological group $H\in\mathscr{S}$ which contains $G$ as a subgroup. A topological group which is $H$-closed in the class of all topological group is called \emph{absolutely closed}. A topological group $G$ is absolutely closed if and only if it
is {\it Ra\v\i kov-complete}, that is complete with respect to the two-sided uniformity which is defined as the least upper bound $\mathcal
L\vee\mathcal R$ of the left and the right uniformities on $G$ \cite{Raikov1946}. Recall
that the sets $\left\{(x,y)\colon x^{-1}y\in U\right\}$, where  $U$  runs  over  a
base  at  unit  of $G$, constitute a base of entourages for the
left uniformity $\mathcal L$ on $G$.  In  the  case  of  the  right
uniformity  $\mathcal R$,  the condition $x^{-1}y\in U$ is replaced by
$yx^{-1}\in U$. The {\it Ra\v\i kov completion $\hat G$\/} of a
topological group $G$ is the completion of $G$ with respect to the two-sided uniformity $\mathcal L\vee\mathcal R$. For every topological group
$G$ the space $\hat G$ has a natural structure of a topological
group. The group $\hat G$ can be defined as a unique (up to an
isomorphism) Ra\v\i kov complete group containing $G$ as a dense
subgroup. There is the following inner characterization of Ra\v\i kov completeness formulated by
the neighbourhood language: a topological group $G$ is Ra\v\i kov complete iff each fundamental
filter $\mathcal F$ on $G$ converges. We recall that a filter $\mathcal F$ on
a topological group $G$ is called {\it fundamental}, provided for each neighbourhood $U$ of the unit
$e$ of the group $G$ there exists a member $F\in\mathcal F$ such that $FF^{-1}\cup F^{-1}F\subset
U$~\cite{Weil-1938}.

Characterizations and investigation methods $H$-closed paratopological groups in the class of abelian paratopological groups turned out to be closely related with
that for minimal Abelian topological groups \cite{BanRavPre1}, \cite{BanRavPre2}, \cite{Ravsky2002} and \cite{Ravsky2003} (in particular, with
famous Prodanov and Stoyanov Theorem about precompactness of such groups). Following classics, Banakh
formulated the following conjecture:
{\it An Abelian topological group $G$ is $H$-closed in the class of paratopological groups if
and only if $G$ is Ra\v\i kov complete and $nG$ is precompact 
for some natural $n$}, where $nG=\left\{na\colon a\in G\right\}$. The conjecture
is proved for some classes of abelian topological groups.
For instance, it holds for a group
$(G,\tau)$ for which there exists a $\sigma$-compact subgroup $L$ of $G$ such that
$G/L$ is periodic {\it and} there exists a group topology $\tau'\subset\tau$ such that
the Ra\v\i kov completion $\hat G$ of the group $(G,\tau')$ is Baire
(see \cite[Proposition~17]{Ravsky2003}). Also the conjecture is true if $(G,\tau)$ is countable, or divisible, or
characters of the group $(G,\tau)$ separate its points and $(G,\tau)$ is \v Cech complete or
periodic \cite{BanRav2001}.

Also, we remark that the problem of the closure of a group in topological and semitopological semigroups was discussed in the books \cite[Vol.~1]{CHK} and \cite{Ruppert1984}. Some new results on this topic can be found in \cite{Gutik-Pavlyk-2003} and \cite{Gutik-2014}.

In this paper we obtained a sufficient condition for a quasitopological group
to be $H$-closed, which allowed us to
solve a problem by Arhangel'skii and Choban and to show that a topological group $G$ is
$H$-closed in the class of quasitopological groups if and only if $G$ is Ra\v\i kov-complete.
We present examples of non-compact quasitopological groups whose topological spaces are $H$-closed.
Also we show that a regular semitopological group with a dense paratopological group is a paratopological group.


We need the following proposition from \cite{ArhangelskiiTkachenko2008}:

\begin{proposition}
[{\cite[Proposition~1.4.13]{ArhangelskiiTkachenko2008}}]\label{proposition-1}
Let $G$ be a quasitopological group and $H$ be a subgroup of $G$. Then closure $\operatorname{cl}_G(H)$ of $H$ in $G$ is a subgroup of $G$.
\end{proposition}

We remark that the closure $\operatorname{cl}_G(H)$ of a subgroup $H$ in a paratopological group (and hence in a semitopological group) $G$ is not necessary a subgroup of $G$ (see: \cite[Example~1.4.17]{ArhangelskiiTkachenko2008}).

By Ellis' Theorem \cite{Ellis1957} every locally compact semitopological group is a topological group.

\begin{proposition}\label{lemma-2}
Every locally compact topological group $G$ is a closed subgroup in any quasitopological group which contains $G$ as a subgroup.
\end{proposition}

\begin{proof}
Suppose that $G$ is a locally compact subgroup of a quasitopological group $H$. Then by Theorem~3.3.9 from \cite{Engelking1989}, $G$ is an open subset of the closure $\operatorname{cl}_H(G)$. Now, Proposition~\ref{proposition-1} implies that the closure $\operatorname{cl}_H(G)$ is a subgroup of the quasitopological group $H$, and hence $\operatorname{cl}_H(G)$ is a quasitopological group. Since every left or right translation in the quasitopological group $\operatorname{cl}_H(G)$ is a homeomorphism of the topological space $\operatorname{cl}_H(G)$ we get that $hG$ is an open subset of $\operatorname{cl}_H(G)$ for each element $h$ in $G$. This implies that $\operatorname{cl}_H(G)\setminus G=\bigcup_{h\in \operatorname{cl}_{H}(G)\setminus G}hG$ is an open subset of the quasitopological group $H$, and hence $G$ is a closed subset of the space  $\operatorname{cl}_H(G)$.
\end{proof}

\begin{remark}
We observe that a counterpart of Proposition~\ref{lemma-2} does not hold for paratopological groups,
because
there exists a Hausdorff paratopological group which contains the discrete group of
integers as a non-closed subgroup (see~\cite[Lemma 3 or Proposition 5]{Ravsky2003}).
\end{remark}

\begin{proposition}\label{proposition-4a}
Every regular semitopological group which contains a dense paratopological group is a paratopological group.
\end{proposition}

\begin{proof}
Suppose that a paratopological group $G$ is a dense subgroup of a regular semitopological group $H$ and $W$ is an open neighbourhood of the unit $e$ of the group $H$. Then the regularity of the space $H$ implies that there exists open neighbourhood $U$ of the unit $e$ in the space $H$ such that $\operatorname{cl}_H(U)\subseteq W$. Since $G$ is a dense paratopological subgroup of $H$, there exists an open neighbourhood $V$ of the unit $e$ in the space $H$ such that
\begin{equation*}
    (V\cap G)\cdot(V\cap G)\subseteq U\cap G.
\end{equation*}
Now the separate continuity of the group operation in $H$ and Theorem~1.4.1 from \cite{Engelking1989} imply that $x\cdot\operatorname{cl}_H(V\cap G)\subseteq \operatorname{cl}_H(U\cap G)$ for every $x\in V\cap G$, so
\begin{equation}\label{eq-1}
    (V\cap G)\cdot\operatorname{cl}_H(V\cap G)\subseteq \operatorname{cl}_H(U\cap G).
\end{equation}
Next, if we apply similar arguments to formula (\ref{eq-1}) then we obtain
\begin{equation*}
    \operatorname{cl}_H(V\cap G)\cdot\operatorname{cl}_H(V\cap G)\subseteq \operatorname{cl}_H\left(\operatorname{cl}_H(U\cap G)\right)=\operatorname{cl}_H(U\cap G).
\end{equation*}
Then Theorem~1.3.6 of \cite{Engelking1989} implies that $\operatorname{cl}_H(V\cap G)=\operatorname{cl}_H(V)$ and so
\begin{equation*}
    V\cdot V\subseteq \operatorname{cl}_H(V)\cdot\operatorname{cl}_H(V)\subseteq \operatorname{cl}_H(U)\subseteq W.
\end{equation*}
This implies the continuity of the group operation in $H$.
\end{proof}

Propositions~\ref{proposition-4a} and ~\ref{proposition-1} imply

\begin{theorem}\label{theorem-6}
A topological group $G$ is $H$-closed in the class of regular quasitopological groups if and only if $G$ is absolutely closed.
\end{theorem}

The following example shows a counterpart of Proposition~\ref{proposition-4a} does not hold for non-regular Hausdorff quasitopological groups.

\begin{example}\label{example-7}
Let $\mathbb{R}^2$ be the direct square of the additive group of real numbers.
Let $\alpha$ be an arbitrary irrational number. Then the line $L(\alpha)=\left\{(x,y)\in\mathbb{R}^2\colon y=\alpha x\right\}$ has the unique point $(0,0)$ with two rational coordinates. Also this property has the line $L(-\alpha)$ in $\mathbb{R}^2$. Put $L=(L(\alpha)\cup L(-\alpha))\setminus\{(0,0)\}$.

We define a topology $\tau$ on the group $\mathbb{R}^2$ in the following way. The family
\begin{equation*}
    \mathscr{B}=\left\{U_{\varepsilon}=B_{\varepsilon}\setminus L\colon \varepsilon>0\right\},
\end{equation*}
where $B_{\varepsilon}=\left\{(x,y)\in \mathbb{R}^2\colon x^2+y^2<\varepsilon^2\right\}$ is a usual $\varepsilon$-ball in $\mathbb{R}^2$, is the base of the topology $\tau$ at the neutral element $(0,0)$ of $\mathbb{R}^2$.

It is easy to check that:
\begin{itemize}
  \item[$(a)$] $(\mathbb{R}^2,\tau)$ with coordinate addition is a quasitopological group;
  \item[$(b)$] the induced topology from $(\mathbb{R}^2,\tau)$ onto $\mathbb{Q}^2$ coincides with the usual topology on $\mathbb{Q}^2$, and hence $\mathbb{Q}^2$ with the coordinate addition is a topological subgroup of $(\mathbb{R}^2,\tau)$;
  \item[$(c)$] $(\mathbb{R}^2,\tau)$ is not a regular topological space;
  \item[$(d)$] $(\mathbb{R}^2,\tau)$ is a functionally Hausdorff space (a topological space $X$ is called \emph{functionally Hausdorff } if for arbitrary distinct point $x$ and $y$ in $X$ there exists a continuous map $f\colon X\to [0,1]$ such that $f(x)=0$  and $f(y)=1$) and hence is Hausdorff.
\end{itemize}
\end{example}

No we are going to formulate a sufficient condition for a quasitopological group
to be $H$-closed in the class of quasitopological groups. We shall call a quasitopological group $G$
{\it Cauchy completty} if each Cauchy filter $\mathcal F$  on $G$ with an base consisting of its
open subsets, has a limit point.
A filter $\mathcal F$ on a quasitopological group $G$ we shall call {\it Cauchy filter},
provided for each neighbourhood $U$ of the unit $e$ of the group $G$ there exists a
member $F\in\mathcal F$ such that $yU\cap Uy\in\mathcal F$ for each point $y\in F$ \cite{Batikova2009}.
Theorem~\ref{HClSp} implies that if $G$ is a quasitopological group which is an
$H$-closed space, then $G$ is Cauchy completty. From the other side we have a following

\begin{proposition}\label{occ} Every Cauchy completty quasitopological group $G$ is $H$-closed in the class of  quasitopological groups.
\end{proposition}

\begin{proof}Suppose to the contrary that there exists a Cauchy completty quasitopological group $G$
which is not $H$-closed in the class of quasitopological groups. By
Proposition~\ref{proposition-1},
without loss of generality we can assume that there exists a quasitopological group $H$
which contains  the group $G$ as a dense proper subgroup. Fix an arbitrary
$x\in H\setminus G$. Put
\begin{equation*}
\mathcal F=\{V_x\cap G\colon V_x\subset H \hbox{~is a neighbourhood of the point~} x \hbox{~in~} H\}.
\end{equation*}
The family $\mathcal F$ is a filter on $G$ with a base
\begin{equation*}
\{V_x\cap G\colon V_x\subset H \hbox{~is an open in~} H \hbox{~ neighbourhood of the point~} x\}
\end{equation*}
which consists of open subsets of the space $G$. We claim that $\mathcal F$ is a Cauchy filter.
Indeed, let
$U$ be an arbitrary open neighbourhood of the unit $e$ of the group $G$.
There exists a symmetric open neighbourhood $V=V^{-1}$ of the unit $e$ of the group $H$ such
that $V\cap G\subset U$.
A set $F=xV\cap Vx\cap G$ is an element of the filter $\mathcal F$. Let $y$ be an arbitrary
point of $F$. Since $y\in xV$ we have that $x\in yV^{-1}=yV$. Similarly, the inclusion
$y\in Vx$ implies that $x\in V^{-1}y=Vy$. So a set $yV\cap Vy$ is an open
neighbourhood of the point $x$ in the space $H$. Then a set $yV\cap Vy\cap G$ is an element of the filter $\mathcal
F$. Since $y\in G$ we conclude that $yV\cap G\subset yU$ and $Vy\cap G\subset Uy$. Therefore
$yV\cap Vy\cap G\subset yU\cap Uy\in\mathcal F$. Hence $\mathcal F$ is a Cauchy filter.
Since the group $G$ is Cauchy completty, the filter $\mathcal F$ has a limit point
$x'\in G$. Since $x\not\in G$ we get that $x\ne x'$. The Hausdorffness of the space of the quasitopological group $H$ implies that
there exists an open neighbourhood $V_x$ of the point $x$ in $H$ such that
$x'\not\in\operatorname{cl}_H(V_x)$. Hence $x'\not\in\operatorname{cl}_G(V_x\cap G)$  which contradicts the following conditions:
$V_x\cap G\in\mathcal F$ and $x'$ is a limit point of filter $\mathcal F$. The obtained
contradiction implies that the quasitopological group $G$ is $H$-closed in the class of quasitopological groups.
\end{proof}

Arhangel'skii and Choban in~\cite{AC} proved that each \v Cech-complete quasitopological group is
$H$-closed in the class of quasitopological groups.
Since each \v Cech-complete semitopological group is a topological group
(see~\cite{Bou} or~\cite{Rez}),
and each \v Cech-complete topological group is Ra\v\i kov complete
(see~\cite[Theorem 4.3.7]{ArhangelskiiTkachenko2008})
this suggested Arhangel'skii and Choban to pose a problem about $H$-closedness of
Ra\v\i kov complete topological group in the class of quasitopological groups.
The next corollary is a positive answer to this problem.

\begin{corollary} A topological group $G$ is $H$-closed in the class of quasitopological groups if
and only if $G$ is Ra\v\i kov complete.
\end{corollary}
\begin{proof} The implication $(\Rightarrow)$ is trivial.

$(\Leftarrow)$ Suppose that $G$ is a Ra\v\i kov complete
topological group. By Proposition~\ref{occ} it suffices to show that the topological group $G$ is
Cauchy completty. For this purpose it suffices to show that each Cauchy filter $\mathcal F$
on the group $G$ is fundamental. Let $U$ be an arbitrary neighbourhood of the unit $e$ of the group $G$.
Since $G$ is a topological group, there exists a neighbourhood $V$ of the unit $e$ of the group $G$
such that $VV^{-1}\cup V^{-1}V\subset U$. Since $\mathcal F$ is a Cauchy filter,
there exists a member $F$ of the filter $F\in\mathcal F$ such that
such that $yV\cap Vy\in\mathcal F$ for each point $y\in F$. Fix an arbitrary point $y\in F$.
Then we get
\begin{equation*}
(yV\cap Vy)(yV\cap Vy)^{-1}\cup (yV\cap Vy)^{-1}(yV\cap Vy)\subseteq
(Vy)(Vy)^{-1}\cup (yV)^{-1}(yV)\subseteq VV^{-1}\cup V^{-1}V\subseteq U,
\end{equation*}
and hence the filter $\mathcal F$ is fundamental.
\end{proof}

\begin{remark}
We observe that in the paper of Arhangel'skii and Choban \cite{AC} there proved  that a \v{C}ech-complete quasitopological group is $H$-closed. It is obvious that our Proposition~\ref{lemma-2} is a simple consequence of this fact. We present this proposition because it has a short proof with the nice idea.
\end{remark}

A family $\mathscr{B}$ of open subsets of a topological space $X$ is called $\pi$-base of $X$ if for every open subset $G$ of $X$ there exists $B\in \mathscr{B}$ such that $B\subseteq G$ \cite{Juhasz-1980}.

\begin{proposition} Let $(G,\tau)$  be a Cauchy completty quasitopological group, $(G,\sigma)$
be a quasitopological group, $\sigma\supset\tau$, and the space $(G,\sigma)$
has a $\pi$-base, consisting of open subsets of the space $(G,\tau)$. Then
$(G,\sigma)$ is a Cauchy completty quasitopological group.
\end{proposition}

\begin{proof}
Let $\mathcal F'$ be a Cauchy filter on the group $(G,\sigma)$ with a base
$\mathcal B'\subset\sigma$.
Put $\mathcal B=\{\operatorname{int}_{(G,\tau)}(F')\colon F'\in\mathcal B'\}$.
Since the space $(G,\sigma)$ has a $\pi$-base which consists of open subsets of the space $(G,\tau)$,
a set $\operatorname{int}_{(G,\tau)}(F')$ is non-empty for any set $F'\in\mathcal B'$.
Therefore $\mathcal B\subset\tau$ is a base of a filter $\mathcal F$ on $G$.
Since $\operatorname{int}_{(G,\tau)}(F')\subset F'$ for each member $F'$ of the base $\mathcal B'$
of the filter $\mathcal F$,  we have that $\mathcal F'\subset\mathcal F$. Then the conditions
$\tau\subset\sigma$, $\mathcal F'\subset\mathcal F$ and $\mathcal F'$ is a Cauchy filter on the group
$(G,\sigma)$ imply that  $\mathcal F$ is a Cauchy filter on the group $(G,\tau)$ too.
Since the base $\mathcal B$ of the Cauchy filter $\mathcal F$ consists of open subsets
of the space $(G,\tau)$ and the group $(G,\tau)$ is Cauchy completty, we conclude that
there is a limit point $x_0$ of the filter $\mathcal F$ in the space $(G,\tau)$.
Let $U=U^{-1}$ be an arbitrary symmetric neighbourhood of the unit of the group
$(G,\sigma)$. Since $\mathcal F'\subset\mathcal F$ and
$\mathcal F'$ is a Cauchy filter on the group $(G,\sigma)$ we get that
there exists an element $F^*\in\mathcal F$ such that $yU\cap Uy\in\mathcal F'\subset\mathcal F$
for each point $y\in F^*$. Fix an arbitrary  $y\in F^*$.
Since $x_0$
is a limit point of the filter $\mathcal F$ in the space $(G,\tau)$ we have that
$x_0\in\operatorname{cl}_{(G,\tau)}(yU\cap Uy)$. Then
$y\in\operatorname{cl}_{(G,\tau)}(x_0U\cap Ux_0)$, hence
$F^*\subseteq\operatorname{cl}_{(G,\tau)}(x_0U\cap Ux_0)$. Since $\mathcal B$ is a base of the filter
$\mathcal F$ we get that there exists a subset $F\in\mathcal B\subset\tau$ such that $F\subset F^*$. Then we have that
$F\cap x_0U\cap Ux_0\ne\varnothing$. Therefore we see that an arbitrary open neighbourhood $U(x_0)\in\sigma$
of the point $x_0$ intersects an element of the filter $\mathcal F$ and  therefore
the neighbourhood $U(x_0)$ intersects an element of the filter $\mathcal F'$, too. Hence we get that $x_0$ is a limit point of the
filter $\mathcal F'$.
\end{proof}

We need the following facts from the paper \cite{RavskyPre1} for the
construction of a non-compact quasitopological group $G$ which is an $H$-closed topological space.

Two topologies $\tau$ and $\sigma$ on the set $X$ we call {\it cowide} provided for any
nonempty sets $U\in\tau$ and $V\in\sigma$ the intersection $U\cap V$ is nonempty too. If a
topology $\sigma$ is cowide to itself then we  call the topology $\sigma$ {\it wide}.

For two topologies  $\tau$ and $\sigma$ on the set $X$ by $\tau\vee\sigma$ we denote
the \emph{supremum} of the topologies $\tau$ and $\sigma$, i.e., $\tau\vee\sigma$ has a base $\{U\cap V\colon U\in\tau,
V\in\sigma\}$.

\begin{lemma}[{\cite[Proposition 18]{RavskyPre1}}]\label{SupHclosed}
Let $\tau$ and $\sigma$ be cowide topologies on the set $X$ such that
the space $(X,\tau)$ is $H$-closed and the topology $\sigma$ is wide. Then the
space $(X,\tau\vee\sigma)$ is $H$-closed too.
\end{lemma}

\begin{lemma}\label{lemma++}
If $(G,\tau)$ and $(G,\theta)$ are quasitopological groups, then $(G,\tau\vee\theta)$ is quasitopological group.
\end{lemma}

\begin{proof}
Fix arbitrary point $x,y\in(G,\tau\vee\theta)$. Let $U^{\tau\vee\theta}(xy)$ be an arbitrary open neighbourhood of the point $xy$ in the space $(G,\tau\vee\theta)$. Then $U^{\tau\vee\theta}(xy)\supseteq U^{\tau}(xy)\cap U^{\theta}(xy)$ for some open neighbourhoods $U^{\tau}(xy)$ and $U^{\theta}(xy)$ of the point $xy$ in the spaces $(G,\tau)$ and $(G,\theta)$, respectively. Then the separate continuity of the group operations in $(G,\tau)$ and $(G,\theta)$ implies that  there exist open neighbourhoods $V^{\tau}(x)$ and $V^{\tau}(y)$ of the points $x$ and $y$, respectively, in the space $(G,\tau)$, and open neighbourhoods $V^{\theta}(x)$ and $V^{\theta}(y)$ of the points $x$ and $y$ in the space $(G,\theta)$, such that $V^{\tau}(x)y\cup xV^{\tau}(y)\subseteq U^{\tau}(xy)$ and $V^{\theta}(x)y\cup xV^{\theta}(y)\subseteq U^{\theta}(xy)$. We put $V^{\tau\vee\theta}(x)= V^{\tau}(x)\cap V^{\theta}(x)$ and $V^{\tau\vee\theta}(y)= V^{\tau}(y)\cap V^{\theta}(y)$.
Then we get that
\begin{equation*}
V^{\tau\vee\theta}(x)y\subseteq (V^{\tau}(x)\cap V^{\theta}(x))y\subseteq V^{\tau}(x)y\cap V^{\theta}(x)y \subseteq U^{\tau}(xy)\cap U^{\theta}(xy)= U^{\tau\vee\theta}(xy)
\end{equation*}
and
\begin{equation*}
xV^{\tau\vee\theta}(y)\subseteq x(V^{\tau}(y)\cap V^{\theta}(y)) \subseteq x V^{\tau}(y)\cap xV^{\theta}(y) \subseteq U^{\tau}(xy)\cap U^{\theta}(xy)= U^{\tau\vee\theta}(xy).
\end{equation*}
Also for any open neighbourhoods $U^{\tau}(e)$ and $U^{\theta}(e)$ of the unit $e$ of the group $G$ in the spaces $(G,\tau)$ and $(G,\theta)$, respectively, we have that $\left(U^{\tau\vee\theta}(e)\right)^{-1}=U^{\tau\vee\theta}(e)$, where $U^{\tau\vee\theta}(e)=U^{\tau}(e)\cap U^{\theta}(e)$. Hence $(G,\tau\vee\theta)$ is a quasitopological group.
\end{proof}

Further on we present two examples of non-compact quasitopological groups which are $H$-closed spaces with different ideas of constructions.

\begin{example}\label{BardEx2+} Let $\mathbb{T}=\left\{z\in\mathbb{C}\colon |z|=1\right\}$ be the unit circle with the operation of usual multiplication of complex numbers and $\mathbb{Q}$ be the additive group of rational numbers. Let $s\colon \mathbb{T}\to\mathbb{Q}$ be any surjective homomorphism.
Such a homomorphism can be build as follows.
Let $\alpha\in\Bbb R$ be an arbitrary irrational number.
Consider the injective isomorphism: $\overline\alpha\colon\Bbb Q\to\Bbb R$, $\overline\alpha:x\mapsto\alpha x$
and the standard homomorphism $q\colon\Bbb R\to\Bbb T$, $q:x\mapsto e^{2\pi xi}$.
Then $q\overline\alpha$ is an isomorphic embedding of a group $\Bbb Q$ into the group $\Bbb T$.
Put $H=q\overline\alpha(\Bbb Q)$.
Since $H$ is a divisible subgroup of the Abelian group $\Bbb T$, by~\cite[$\S$23]{Kur} $H$ is
a direct summand of the group $\Bbb T$. So the isomorphism $(q\overline\alpha)^{-1}:H\to\Bbb Q$ can be extended to a
homomorphism $s:\Bbb T\to\Bbb Q$.

For each positive integer $n$ we put $U_n=\left\{x\in\mathbb{T}\colon |s(x)|>n\right\}$.
It is easy to check that the family $\left\{\{x\}\cup (xU_n)\colon x\in\mathbb{T}, n\in\mathbb{N}\right\}$ is a base of a
wide topology $\sigma$ at the point $x$ of the group $\mathbb{T}$.  Let $\tau$ be the usual topology on $\mathbb{T}$. Then the topological space $(\mathbb{T},\tau)$ is compact
and therefore $H$-closed. Since each subset $U_n$ is dense in $(\mathbb{T},\tau)$, topologies $\tau$ and $\sigma$ are cowide. It is obvious that the space $(\mathbb{T},\tau\vee\sigma)$ is not compact.
Then Lemma~\ref{SupHclosed} implies that $(\mathbb{T},\tau\vee\sigma)$ is an $H$-closed space, and it is easy to check that $(\mathbb{T},\sigma)$ is a quasitopological group. Then  $(\mathbb{T},\tau\vee\sigma)$ is a quasitopological group
too.
\end{example}

\begin{example}\label{ex-2}
Let the group $(\mathbb{T},\tau)$ be as in Example~\ref{BardEx2+} and $\sigma$ be the topology on $\mathbb{T}$ generated by the base which consists of co-countable subsets of $\mathbb{T}$. It easy to check that  $(G,\tau)$ and $(G,\sigma)$ are quasitopological groups. Lemma~\ref{lemma++} implies that $(\mathbb{T},\sigma\vee\tau)$ is a quasitopological group. Next we show that $(\mathbb{T},\sigma\vee\tau)$ is an $H$-closed topological space. We remark that $\left\{A\cap B\ \colon A\in \tau, B\in \sigma\right\}$ is a base of the topology $\sigma\vee\tau$ on $\mathbb{T}$.
Suppose that $\left\{U_{\alpha}\colon \alpha\in\mathscr{I}\right\}$ is an open cover of the topological space $(\mathbb{T},\sigma\vee\tau)$ which consists of basic open subsets. Then for every $\alpha\in\mathscr{I}$ we have that $U_{\alpha}= A_{\alpha}\cap B_{\alpha}$, where $A_{\alpha}$ and $B_{\alpha}$ are elements of the topologies $\tau$ and $\sigma$, respectively. Now, the compactness of the space $(G,\tau)$ implies that there exists a finite subfamily $\left\{A_{\alpha_{1}},.., A_{\alpha_{n}}\right\}$ in $\left\{A_{\alpha}\colon \alpha\in\mathscr{I}\right\}$, such that $\bigcup^{n}_{i=1}A_{\alpha_{i}}= G$. Then the definitions of topologies $\tau$ and $\sigma$ imply that $U_{\alpha_{i}}$ is a dense subset of $A_{\alpha_{i}}$ in the topological space $(\mathbb{T},\sigma\vee\tau)$ for any $i=1,\ldots,n$. This implies that $\bigcup^{n}_{i=1}\operatorname{cl}_{(\mathbb{T},\sigma\vee\tau)}\left(U_{\alpha_{i}}\right)= G$ and hence by Theorem~\ref{HClSp} $(\mathbb{T},\sigma\vee\tau)$ is an $H$-closed topological space.
\end{example}

\section*{Acknowledgements}

We acknowledge the referee for his (her) important  several comments and suggestions.

\end{document}